\documentclass[10pt,reqno]{amsart}
\usepackage{amssymb,amsmath,amscd}
\usepackage[top=3.1cm, bottom=3.1cm, left=3.2cm, right=3.2cm]{geometry}

\usepackage[all]{xy}
\usepackage{graphicx}
\usepackage{calc}

\newcommand{\bigboxplus}{
  \mathop{
    \vphantom{\bigoplus} 
    \mathchoice
      {\vcenter{\hbox{\resizebox{\widthof{$\displaystyle\bigoplus$}}{!}{$\boxplus$}}}}
      {\vcenter{\hbox{\resizebox{\widthof{$\bigoplus$}}{!}{$\boxplus$}}}}
      {\vcenter{\hbox{\resizebox{\widthof{$\scriptstyle\oplus$}}{!}{$\boxplus$}}}}
      {\vcenter{\hbox{\resizebox{\widthof{$\scriptscriptstyle\oplus$}}{!}{$\boxplus$}}}}
  }\displaylimits 
}
\newcommand{\bigast}{
  \mathop{
    \vphantom{\bigoplus} 
    \mathchoice
      {\vcenter{\hbox{\resizebox{\widthof{$\displaystyle\bigoplus$}}{!}{$\ast$}}}}
      {\vcenter{\hbox{\resizebox{\widthof{$\bigoplus$}}{!}{$\ast$}}}}
      {\vcenter{\hbox{\resizebox{\widthof{$\scriptstyle\oplus$}}{!}{$\ast$}}}}
      {\vcenter{\hbox{\resizebox{\widthof{$\scriptscriptstyle\oplus$}}{!}{$\ast$}}}}
  }\displaylimits 
}

\theoremstyle{plain}
\newtheorem{theorem}{Theorem}[subsection]                    
\newtheorem{proposition}[theorem]{Proposition}            
\newtheorem{corollary}[theorem]{Corollary}                
\newtheorem{lemma}[theorem]{Lemma}

\theoremstyle{definition}
\newtheorem{remark}[theorem]{Remark} 
\newtheorem{definition}[theorem]{Definition}

\def\us#1_#2{\underset{#2}{#1}}
\def\os#1^#2{\overset{#2}{#1}}
\def\i<#1>{\langle {#1} \rangle}

\DeclareMathOperator{\Imag}{Im}

\newcommand{\cb}{{\mathord{\rm cb}}}
\newcommand{\red}{{\mathord{\rm red}}}
\newcommand{\too}{\longrightarrow}
\newcommand{\ol}{\overline}
\newcommand{\DT}{\Delta\bfT}
\newcommand{\dT}{\partial\bfT}
\newcommand{\act}{\curvearrowright}
\DeclareMathOperator{\lspan}{span}
\DeclareMathOperator{\ospan}{\overline{\lspan}}
\DeclareMathOperator{\Ad}{Ad}
\DeclareMathOperator{\id}{id}

\makeatletter 
\@tempcnta\z@
\loop\ifnum\@tempcnta<26
\advance\@tempcnta\@ne
\expandafter\edef\csname frak\@Alph\@tempcnta\endcsname{\noexpand\mathfrak{\@Alph\@tempcnta}}
\expandafter\edef\csname l\@Alph\@tempcnta\endcsname{\noexpand\mathbb{\@Alph\@tempcnta}}
\expandafter\edef\csname cal\@Alph\@tempcnta\endcsname{\noexpand\mathcal{\@Alph\@tempcnta}}
\expandafter\edef\csname rm\@Alph\@tempcnta\endcsname{\noexpand\mathrm{\@Alph\@tempcnta}}
\expandafter\edef\csname bf\@Alph\@tempcnta\endcsname{\noexpand\mathbf{\@Alph\@tempcnta}}
\repeat
\makeatother

\allowdisplaybreaks[4]
\begin{document}
\title[Bass--Serre trees]{Bass--Serre trees of amalgamated free product $\rmC^*$-algebras}
\author[K.~Hasegawa]{Kei~Hasegawa}
\address{Graduate~School~of~Mathematics, Kyushu~University, Fukuoka~819-0395, Japan}
\email{k-hasegawa@math.kyushu-u.ac.jp}
\begin{abstract}
For any reduced amalgamated free product $\mathrm{C}^*$-algebra $(A,E)=(A_1, E_1) \ast_D (A_2,E_2)$,
we introduce and study a canonical ambient $\mathrm{C}^*$-algebra $\Delta\mathbf{T}(A,E)$ of $A$
which generalizes the crossed product  arising from the canonical action of an
amalgamated free product group on the compactification of the associated Bass--Serre tree.
Using an explicit identification of $\DT(A,E)$ with a Cuntz--Pimsner algebra
we prove two kinds of ``amenability'' results for $\DT(A,E)$; nuclearity and universality.
As applications of our framework, we 
provide new conceptual, and simpler proofs of several known theorems
on approximation properties, embeddability, and $KK$-theory for reduced amalgamated free product $\mathrm{C}^*$-algebras.
\end{abstract}
\maketitle

%
%
\section{Introduction}\label{section_intro}\renewcommand{\thetheorem}{\Alph{theorem}}
The reduced amalgamated free product of $\rmC^*$-algebras
was introduced by Voiculescu in \cite{Voiculescu}.
Since then, various aspects of the construction,
for example, simplicity, stable rank \cite{Dykema_sr,Dykema_Haagerup_Rordam},
approximation properties \cite{Dykema,Dykema_Shly,Ricard_Xu,Ozawa_RIMS},
and $K$-theory \cite{Hasegawa_IMRN,Fima_Germain1}
have been studied.
In this paper, we are interested in a new aspect of the construction,
namely, an analogue of the boundary actions arising from the Bass--Serre trees
of amalgamated free product groups.

The Bass--Serre tree associated with an amalgamated free product $\Gamma=\Gamma_1\ast_\Lambda\Gamma_2$
is a tree $\bfT$ on which $\Gamma$ acts canonically,
and such an action is a fundamental example of group actions on trees in Bass--Serre theory \cite{Serre}.
In \cite{Bowditch}, Bowditch introduced the compactification $\DT$
for a more general class of hyperbolic graphs
and studied the induced action $\Gamma \act \DT$.
The $\Gamma$-space $\DT$ can be viewed as an analogue of Gromov boundary for hyperbolic groups;
it (or its suitable quotient) is a $\Gamma$-boundary in the sense of Furstenberg \cite{Furstenberg}
and the amenability of the action was studied by Ozawa \cite{Ozawa_TA}.

\medskip
In the present paper,
we introduce and study a unital $\rmC^*$-algebra $\DT(A,E)$
for a given reduced amalgamated free product
$(A,E)=(A_1,E_1)\ast_D(A_2,E_2)$.
The $\rmC^*$-algebra $\DT(A,E)$ is generated by
a copy of $A$ and two projections,
and satisfies that
when $(A,E)$ comes from the reduced group
$\rmC^*$-algebra $\rmC^*_\red(\Gamma)$ of $\Gamma=\Gamma_1\ast_\Lambda \Gamma_2$,
one has
\[
(\rmC^*_\red(\Gamma) \subset \DT(\rmC^*_\red(\Gamma) ,E) )\;\cong \; (\rmC^*_\red(\Gamma)\subset C(\DT)\rtimes_\red\Gamma).
\]
Our construction is inspired by the recent works \cite{Hasegawa_IMRN,Fima_Germain1}
on the $K$-theory of amalgamated free products.
In these works, the $KK$-equivalence between
any reduced amalgamated free product and the corresponding full one
is proved.
One of key ingredients in the proof is a geometric description of the $KK$-class
of $\id\colon A\to A$ using
an analogue of the Bass--Serre tree,
inspired by Julg and Valette's work \cite{Julg_Valette} in the group case
and its quantum group analogue \cite{Vergnioux}.
Fima and Germain \cite{Fima_Germain1} then established
six-term exact sequences relating $KK$-groups of $A$ to those of $A_1, A_2,$ and $D$.
These considerations are further generalized to the framework of graph of $\rmC^*$-algebras \cite{Fima_Freslon,Fima_Germain2}.

\medskip
Our main technical result is an explicit identification of $\DT(A,E)$ with a Cuntz--Pimsner algebra (Theorem \ref{prop_cuntz_pimsner})
of which the Toeplitz extension is always semisplit (Theorem \ref{prop_semisplit})
inspired by \cite{Spielberg,Okayasu}.
As a consequence, we obtain two kinds of ``amenability'' results for $\DT(A,E)$.
Our first amenability result is that 
$\DT(A,E)$ is nuclear (resp.,~exact, and has the completely bounded approximation property) if and only if $A_1$ and $A_2$ have the same property (Corollary \ref{cor_app}).
This follows from the well-established theory of Pimsner algebras \cite{Pimsner_free,Katsura,Dykema_Smith},
and partially generalizes Ozawa's result \cite{Ozawa_TA}.
The other one is that
$\DT(A,E)$ is
the universal $\rmC^*$-algebra generated by
a unital copy of the \emph{algebraic} amalgamated free product of $A_1$ and $A_2$ over $D$
and projections $e_1$ and $e_2$ such that $e_1+e_2=1$ and
\[
e_k a e_k =E_k(a)e_k \quad \text{for}\quad a\in A_k, \; k=1,2
\]
(Corollary \ref{cor_univ}).
This universality can be viewed as
an analogue of the
isomorphisms between the full and reduced crossed products for amenable actions.
Since the definition of the universal $\rmC^*$-algebra does not involve the reduced amalgamated free product,
it provides a new criterion for $*$-homomorphisms from the full amalgamated free product factoring through the reduced one.

\medskip
As applications,
we give simple proofs of the following known results by using the universality and the Cuntz--Pimsner algebra structure of $\DT(A,E)$.
\begin{itemize}
\item Reduced amalgamated free product $\rmC^*$-algebras of exact $\rmC^*$-algebras are exact
(Dykema \cite{Dykema}, see \cite{Dykema_Shly,Ricard_Xu} for different proofs).
\item Reduced amalgamated free product $\rmC^*$-algebras of nuclear $\rmC^*$-algebras are nuclear
provided that the image of one of free components under the GNS representation contains the Jones projection (Ozawa \cite{Ozawa_RIMS}).
\item
If each free component of a reduced amalgamated free product admits
an embedding into a corresponding one of another reduced amalgamated free product
that is compatible with conditional expectations,
then the embeddings induce an inclusion of reduced amalgamated free products
(Blan\-chard--Dykema \cite{Blanchard_Dykema}).
\item There exist six-term exact sequences relating $KK$-groups of any reduced amalgamated free product $\rmC^*$-algebra
to those of free compents and the amalgamated subalgebra (Fima--Germain \cite{Fima_Germain1}).
\end{itemize}
We point out that our proof of Ozawa's result for nuclearity works for completely bounded approximation property too.
Also, we note that
in our proof of the exact sequences of $KK$-groups for $(A,E)$,
the semisplit Toeplitz extension of the Cuntz--Pimser algebra
identified with  $\DT(A,E)$ plays a key role.
In the course of our proof,
we will show that
the embedding $A\hookrightarrow \DT(A,E)$ is right invertible in $KK$-theory using the description of $\id \in KK(A,A)$ from \cite{Hasegawa_IMRN,Fima_Germain1}.
Then, the desired sequences will follow from the six-term exact sequences of $KK$-groups (\cite{Cuntz_Skandalis}) induced from the Toeplitz extension.
As a by-product of our approach,
we show that $\DT(A,E)$ is $KK$-equivalent to $A\oplus D$.
By the result in \cite{Hasegawa_IMRN,Fima_Germain2},
this implies that the $KK$-class of $\DT(A,E)$ is independent of the choice
of conditional expectations.

The paper is organized as follows.
In \S \ref{section_prelim} we recall basic facts on Hilbert $\rmC^*$-modules, reduced amalgamated free products,
and Pimsner algebras.
In \S \ref{section_cpt} we construct the $\rmC^*$-algebra $\DT(A,E)$ based on the group case.
The identification result with a Cuntz--Pimsner algebra is given in \S \ref{section_pimsner}.
In \S \ref{section_app} we give proofs of above known results except the last one.
In the final section, we discuss the $K$-theory of $\DT(A,E)$ and give a proof of Fima--Germain's exact sequences of $KK$-groups.

\section*{Acknowledgment}
The author is grateful to Yuhei Suzuki for fruitful discussions on Lemma \ref{lem_nonfull}
and to Narutaka Ozawa for informing him of \cite{Choda}.
He wishes to appreciate his supervisor, Yoshimichi Ueda, for his constant encouragement
and variable comments.
Also, he wishes to thank Gilles Pisier and Texas A\&M University, where this paper was completed,
for the invitation and hospitality.
This work was supported by the Research Fellow of the Japan Society for the Promotion of Science.

\section{Preliminaries}\label{section_prelim}
\renewcommand{\thetheorem}{\arabic{section}.\arabic{subsection}.\arabic{theorem}}
\subsection{Hilbert $\rmC^*$-modules}
We refer to \cite{Lance} for the theory of Hilbert $\rmC^*$-modules.
For a Hilbert $\rmC^*$-module $X$ over a $\rmC^*$-algebra $A$, we denote by $\lL(X)$ the $\rmC^*$-algebra
of all adjointable operators on $X$.
The ``rank one operator'' associated with $\xi, \eta \in X$ is denoted by
$\theta_{\xi, \eta}$ and the elements of
$\lK (X ) := \ospan \{ \theta_{\xi,\eta} \mid \xi,\eta \in X \}$ are called \emph{compact} operators.
An \emph{$A$-$B$ $\rmC^*$-correspondence} is a pair $(Y,\phi_Y)$ such that
$Y$ is a Hilbert $B$-module and $\phi_Y \colon A \to \lL(Y)$ is a $*$-homomorphism.
When $\phi_Y$ is injective, the pair is called \emph{injective}.
we denote the interior tensor product of $X$ and $(Y,\phi_Y)$ by $X\otimes_A Y$.
For each $x \in \lL (X)$,
we denote by $x\otimes 1$ the image of $x$ under the natural map $\lL (X) \to \lL (X\otimes_A Y)$.
Let $\calI$ be a set and $X_i$ be a Hilbert $A_i$-module for $i\in\calI$.
Let $A=\prod_{i\in\calI}A_i$ be the direct product.
We will use the Hilbert $A$-module $\bigboxplus_{i\in\calI}X_i =\bigoplus_{i\in\calI}X_i\otimes_{A_i} A$,
where the interior tensor product is with respect to the natural inclusion $A_i \to A$.
Note that $\lL (\bigboxplus_{i\in\calI} X_i) \cong \prod_{i\in\calI}\lL (X_i)$.
When $\calI=\{1,2\}$, we write $X_1\boxplus X_2$ instead of $\bigboxplus_{i=1,2}X_i$.
The next lemma will be used later.
\begin{lemma}\label{lem_cpt_tensor}
Let $A$ and $B$ be $\rmC^*$-algebras, $X$ be Hilbert $A$-module and $(Y,\phi_Y)$ be an injective $A$-$B$ $\rmC^*$-correspondence.
For any $x\in \lL(X)$, if $x\otimes 1 \in \lL(X\otimes_A Y)$ is compact, then so is $x$.
\end{lemma}
\begin{proof}
Take an approximate unit $(e_i)_i$ of $\lK(X)$.
Since $\|e_i \xi - \xi \| \to 0$ holds for $\xi \in X$,
if $x\otimes 1$ is compact, then $e_i x\otimes 1$ converges to $x\otimes 1$ in norm.
Since $\phi_Y$ is injective, this implies that $x =\lim_i e_ix \in \lK(X)$.
\end{proof}

\subsection{Reduced amalgamated free products}\label{ss_amal}
We fix notations on reduced amalgamated free products used throughout this paper.
Let $\{ (D \subset A_k, E_k )\}_{ k \in \calI}$ be a family of unital inclusions of $\rmC^*$-algebras with
nondegenerate (i.e., the associated GNS representation is faithful) conditional expectations.
For each $m \in \lN$,
we set $\calI_m:= \{ \iota : \{ 1, \dots, m \} \to \calI \mid
\iota (j) \neq \iota (j+1) \text{ for }1 \leq j \leq m-1 \}$.
The \emph{reduced amalgamated free product} $(A, E) = \bigast_{D} (A_k , E_k ) $
is a $\rmC^*$-algebra $A$ generated by copies of $A_k$, $k\in\calI$
such that the embeddings $A_k\hookrightarrow A,k\in\calI$ coincide on $D$
together with a nondegenerate conditional expectation $E$ such that
$
E (a_1 \cdots a_m ) = 0
$
for all $m \in \lN$, $\iota \in \calI_m$ and $a_j \in \ker E_{\iota(j)}^\circ$ for $j= 1,\dots,m$.
(see \cite{Voiculescu} for the construction).
Any element in $A$ of the form $a_1 \cdots a_m$ as above is called a \emph{reduced word}.

We denote by $( X_k , \phi_{X_k}, \xi_k )$ and $(X, \phi_X, \xi_0 )$
the GNS representations associated with $E_k$ and $E$, respectively,
and set $A_k^\circ = \ker E_k$, $X_k^\circ =X_k \ominus \xi_k D$
and $a^\circ := a - E (a)$ for $a\in A$.
Then, $X$ is naturally identified with
\begin{align*}
  \xi_0 D \oplus \bigoplus_{m \geq 1} \bigoplus_{ \iota \in \calI _m }
  X_{\iota ( 1)}^\circ \otimes_D \cdots \otimes_D X_{\iota (m)}^\circ.
\end{align*}
For each $k \in \calI$, we denote by $P_{(\ell, k)}$ and $P_{(r,k)}$ the projections onto the following submodules, respectively:
\begin{align*}
X (\ell, k) &:= \xi_0 D \oplus \bigoplus_{ m \geq 1}  \underset{\iota (1) \neq k }{\bigoplus_{ \iota \in \calI_m}}
  X_{\iota (1)}^\circ \otimes_D \cdots  \otimes_D X_{\iota (m )}^\circ , \\
X (r, k) & :=  \xi_0 D \oplus \bigoplus_{ m \geq 1} \underset{\iota (m) \neq k }{\bigoplus_{ \iota \in \calI_m}}
  X_{\iota (1)}^\circ \otimes_D \cdots  \otimes_D X_{\iota (m )}^\circ .
\end{align*}
The compression by the projection onto $\xi_0D \oplus X_k^\circ \cong X_k$ defines the conditional
expectation $E_{A_k} \colon A \to A_k$ such that $E = E_k\circ E_{A_k}$,
and the associated GNS representation is denoted by $(Y_k, \phi_{Y_k}, \eta_k )$.
Recall that there exists a unitary $S_k \colon X ( r, k ) \otimes_D A_k \to Y_k$ such that
$S_k a_1 \cdots a_m \xi_0 \otimes b  = a_1 \cdots a_m \eta_k b$ holds for all $m\in\lN$,
all reduced words $a_1 \cdots a_m$ with $a_m \notin A_k^\circ$ and $b \in A_k$
(see \cite[Lemma 3.1.1]{Hasegawa_IMRN}).
We use the following $A$-$\prod_{k\in\calI}A_k$ and $A$-$A$ $\rmC^*$-correspondences
\begin{equation}\label{eq_YZ}
  (Y, \phi_Y) = \bigboxplus_{k\in \calI} (Y_k, \phi_{Y_k} ), \quad (Z,\phi_Z ) = \bigoplus_{k\in\calI}(Y_k\otimes_{A_k}A, \phi_{Y_k}\otimes 1),
\end{equation}
and the $*$-homomorphism $\widetilde{\phi}_Z\colon \lL(Z)\to\lL(Y)$
induced from the natural $*$-homomorphisms $\lL(Y_k) \to \lL(Y_k \otimes_{A_k} A)$, $k\in \calI$.
We may identify $A$ with $\phi_Y (A)$.

\begin{remark}
In \cite{Fima_Germain1}, Fima and Germain introduced the vertex-reduced amalgamated
free product to deal with the case when $E_k$ is possibly ``degenerate''.
We note that our results in \S \ref{section_cpt} and \S \ref{section_pimsner}
can be shown for the vertex-reduced amalgamated free product in the same manner.
\end{remark}
\subsection{Pimsner algebras}
We fix notations and terminologies on Pimsner algebras following \cite{Katsura}.
Let $(X, \phi_X)$ be a $\rmC^*$-correspondence over a $\rmC^*$-algebra $A$.
We do \emph{not} assume that $\phi_X$ is injective.
Recall that a \emph{representation} of $X$ on a $\rmC^*$-algebra $B$ is a pair $(\pi, t )$ such that
$\pi \colon A \to B$ is a $*$-homomorphism and $t \colon X \to B$ is a linear map satisfying
$t (\xi)^* t (\eta ) = \pi (\i<\xi, \eta >)$ and
$\pi (a) t (\xi ) \pi (b) = t (\phi_X (a) \xi b )$ for $\xi, \eta\in X$ and $a, b \in A$.
We denote by $\rmC^*(\pi,t)$ the $\rmC^*$-algebra generated by $\pi(A)$ and $t(X)$ inside $B$.

For each $n\in \lN$, let $X^{\otimes n }$ denote the interior tensor product
$X \otimes_A X \otimes_A\cdots \otimes_A X$ of $n$ copies of $X$, equipped with the left action $\phi_X\otimes 1$,
and set $X^{\otimes 0}=A$, on which $A$ acts by the left multiplication.
The full Fock space $\calF (X) = \bigoplus_{n \geq 0} X^{\otimes n}$ over $X$ forms
a $\rmC^*$-correspondence over $A$
and the left action of $A$ on $\calF(X)$ is denoted by $\varphi_\infty$.
For each $\xi$, we define the creation operator $\tau_\infty(\xi)$ on $\calF(X)$ by
$\tau_\infty (\xi ) \eta = \xi \otimes \eta $ for $\eta \in X^{\otimes n}$ and $n\geq 0$.
Then the pair $(\varphi_\infty, \tau_\infty )$ is a representation of $X$
and we call $\calT(X):=\rmC^*(\varphi_\infty,\tau_\infty)$
the \emph{Toeplitz--Pimsner} algebra of $X$.
Note that $(\varphi_\infty,\tau_\infty)$ satisfies the universal property
that
if $(\pi, t)$ is a representation of $X$, then there exists a $*$-homomorphism $\rho\colon \rmC^*(\varphi_\infty,\tau_\infty)
\to \rmC^*(\pi,t)$ such that $\rho\circ \varphi_\infty = \pi$ and $\rho\circ\tau_\infty=t$.
The compression map by the projection onto $X^{\otimes0}$ defines a nondegenerate conditional expectation
$E_X \colon \calT (X) \to A$.

Any representation $(\pi,t)$ induces a $*$-homomorphism $\psi_t \colon \lK (X) \to B$ such that
$\psi_t (\theta_{\xi,\eta} ) = t (\xi)t (\eta)^*$.
We define the ideal $J_X$ of $A$ by
\[
  \phi_X^{-1} (\lK (X)) \cap (\ker \phi_X )^\perp = \{a\in  \phi_X^{-1} (\lK (X)) \mid ax =0 \; \text{for}\; x \in \ker \phi_X\}.
\]
We say that $(\pi, t )$ is \emph{covariant} if $\pi = \psi_t \circ \phi_X$ holds on $J_X$.
The ideal of $\calT(X)$ generated by $\{\varphi_\infty(x) - \psi_{\tau_\infty}(\phi_X(x)) \mid x \in J_X\}$
is naturally isomorphic to $\lK(\calF(X)J_\frakX)$
and the quotient of $\calT(X)$ by $\lK(\calF(X)J_\frakX)$ is called the \emph{Cuntz--Pimsner algebra} of $X$
and denoted by $\calO(X)$.
Note that the definition of $\calO(X)$ is different from Pimsner's original one in \cite{Pimsner_free}.
The covariant representation of $X$ on $\calO(X)$ induced from the quotient map and $(\varphi_\infty,\tau_\infty)$ is
universal among covariant representations.
A representation $(\pi,t)$ is said to \emph{admit a gauge action} if there exists a continuous action $\gamma$
of $\lT= \{z\in\lC \mid |z| =1\}$ on $\rmC^*(\pi,t)$
such that $\gamma_z \circ \pi =\pi$ and $\gamma_z(t(\xi))=z t(\xi)$ for $z \in \lT$ and $\xi \in X$.
We will use the next gauge-invariant uniqueness theorems.

\begin{theorem}[{\cite[Theorem 6.2, Theorem 6.4]{Katsura}}]\label{thm_GIU}
For any representation $(\pi,t)$ of $X$,
the canonical surjection from $\calT(X)$ onto $\rmC^*(\pi,t)$ is injective if and only if
$\pi$ is injective, $(\pi,t)$ admits a gauge action, and $\pi(J_X)\cap \psi_t(\lK(X))=\{0\}$.
If $(\pi,t)$ is covariant,
then the canonical surjection from $\calO(X)$ onto $\rmC^*(\pi,t)$ is injective if and only if
$\pi$ is injective and $(\pi,t)$ admits a gauge action.
 \end{theorem}
\section{Compactifications of Bass--Serre trees}\label{section_cpt}
\subsection{Group case}\label{ss_groupcase}
Let $\Gamma= \Gamma_1 \ast_\Lambda \Gamma_2$ be an amalgamated free product of discrete groups and set $\calI=\{1,2\}$.
Recall that the \emph{Bass--Serre tree} associated with $\Gamma$ is the graph $\bfT$
of which the vertex set is $\bfV =\Gamma / \Gamma_1 \sqcup \Gamma/ \Gamma_2$
and the edge set is $\bfE=\Gamma / \Lambda$,
such that the edge $g \Lambda$ relates $g\Gamma_1$ and $g\Gamma_2$ (see \cite{Serre}).
Then, the \emph{Bowditch compactification} $\Delta \bfT:= \bfV \sqcup \partial \bfT $ of $\bfT$
is given in the following way (see \cite{Bowditch} or \cite[\S\S 5.2]{Brown_Ozawa} for details).
If we identify each $v\in\bfV$ with the unique finite geodesic path in $\bfT$
connecting $v$ with the origin $e\Gamma_1 \in \bfV$ (the coset of $\Gamma_1$ with respect to the neutral element),
then each element in the Gromov boundary $\partial \bfT$ is identified with a one-sided infinite geodesic paths from $e\Gamma_1$.
For each $x, y \in \Delta \bfT$ we denote by $[x, y] \subset \DT $ the unique geodesic path connecting $x$ and $y$,
and for each finite subset $F $ of $V$, we set
\[
M (x, F) := \{ y \in \Delta \bfT \mid [x, y ] \cap F = \emptyset \}.
\]
Then, $M (x, F)$ is clopen in $\DT$ and the family
$\{ M (x, F) \mid x \in \Delta \bfT, F \subset \bfV \text{ finite} \}$ forms an open base for $\Delta \bfT$.
We have natural inclusions $C (\Delta \bfT ) \subset \ell^\infty (\bfV )\subset \lB (\ell^2 (\bfV ))$
and the action $\alpha \colon \Gamma \act C (\DT)$ is implemented by
the unitary representation $\pi \colon \Gamma \act \ell^2 (\bfV)$
induced from the left multiplication $\Gamma \act \bfV$.
Note that $(\ell^2(\bfV), \pi)$ is unitarily equivalent to the direct sum of quasiregular representations $( \ell^2 (\Gamma/\Gamma_k), \lambda_{\Gamma/ \Gamma_k})$, $k=1,2$.

\medskip
Let $A$, $A_k$ and $D$ be the reduced group $\rmC^*$-algebras of $\Gamma$, $\Gamma_k$ and $\Lambda$, respectively,
and $E\colon A \to D$, $E_k \colon A_k \to D$ and $E_{A_k} \colon A\to A_k$ be canonical conditional expectations.
Then, it follows that $(A, E) \cong (A_1, E_1) \ast_D (A_2, E_2)$.
We denote the left regular representation of $\Gamma$ by $\lambda$.
We use the notations in \S\S\ref{ss_amal} and assume that $A \subset \lL (Y)$.
Let $P_k\in\lL(Y)$ be the projection onto the closed submodule generated by
all vectors of the form $a_1a_2\cdots a_n \eta_j$ for some $j\in \calI$, $n\in\lN$, and reduced word $a_1a_2\cdots a_n$ with $a_1 \in A_k^\circ$.
More precisely, set
\begin{equation}\label{eq_P_k}
  P_k^\circ := \sum_{j\in\calI} S_j (P_{(\ell,k)}^\perp P_{(r,k)}\otimes 1 )S_j^*,
  \quad\quad P_k := e_{A_k} + P_k^\circ,
\end{equation}
where $e_{A_k} = \theta_{\eta_k,\eta_k} \in \lK (Y_k ) \subset \lK (Y)$ is the Jones projection of $E_{A_k}$.
\begin{proposition}\label{prop_group}
  The reduced crossed product $C ( \Delta \bfT ) \rtimes_\red \Gamma$ is isomorphic to $\rmC^* ( A, P_1, P_2)$.
\end{proposition}
\begin{proof}
Let $P_{\Gamma_1}$ be the characteristic function of $M(e\Gamma_1, \{e\Gamma_2\})$ and set $P_{\Gamma_2}=1-P_{\Gamma_2}$.
Note that $P_{\Gamma_k}\ell^2(\bfV) $ is the closed linear span of all the vectors $\delta_{g\Gamma_j}$
such that $j =1,2$ and $g$ is a reduced word beginning with some element in $\Gamma_k \setminus \Lambda$.
For any $k\neq j\in\{1,2\}$ and reduced word $g=g_1g_2\cdots g_n$
with $g_n \in \Gamma_j$,
the characteristic function of $M(gP_k,\{gP_j\})$ is equal to $\alpha_g(P_{\Gamma_k})$.
Since these projections separate of the points of $\DT$,
it follows from the Stone--Weierstrass theorem that
$C(\DT)\rtimes_\red \Gamma$ is isomorphic to the $\rmC^*$-subalgebra of $\lL(\ell^2(\bfV)\otimes A)$ generated by
$\{P_{\Gamma_k} \}_{k=1,2}$ and $\{\pi (g) \otimes \lambda (g) \}_{g\in \Gamma}$.
%
%
For each $k\in\{1,2\}$,
we define the unitary $U_k \colon \ell^2 (\Gamma / \Gamma_k) \otimes A \to Y_k  \otimes_{A_k} A$
by $U_k ( \delta_{g \Gamma_k } \otimes a ) = \lambda (g) \eta_k  \otimes \lambda (g )^* a$ for $g \in \Gamma$ and $a \in A$.
Then, one can easily check that $U_k (\lambda_{\Gamma / \Gamma_k} ( g ) \otimes \lambda (g) )
 = ( \phi_{Y_k} ( \lambda (g) ) \otimes 1 ) U_k$ for $g \in \Gamma$.
Let $(Z, \phi_Z )$ and $\widetilde{\phi}_Z \colon \lL (Y) \to \lL (Z)$ be as in Eq.\,(\ref{eq_YZ}).
Letting $U := U_1 \oplus U_2  \colon \ell^2 (V) \otimes A \to Z $ we obtain
\[
  U (\pi (g )  \otimes \lambda (g) ) U^* =\phi_Z ( \lambda (g)) \quad \text{for }g \in \Gamma,
  \quad
  \widetilde{\phi}_Z (P_k )= UP_{\Gamma_k} U^* \quad \text{for }k=1,2.
\]
Since $\widetilde{\phi}_Z$ is injective, $C (\Delta \bfT ) \rtimes_\red \Gamma$ is isomorphic to $\rmC^* ( A, P_1, P_2)$.
\end{proof}

%
%
%
%
\subsection{General case}\label{ss_general}
\begin{definition}
Let $(A, E) = \bigast_{D} (A_k , E_k ) $ be the reduced amalgamated free product of $\{(D\subset A_k, E_k)\}_{k\in\calI}$
and $P_k \in \lL(Y)$ be as in Eq.~(\ref{eq_P_k}).
We define $\DT(A,E)$ by the $\rmC^*$-algebra generated by $\phi_Y(A)$ and $\{P_k\}_{k\in\calI}$.
\end{definition}

We may identify $A$ with $\phi_Y(A)$ and assume that $A\subset \DT(A,E)$.
We will use the following representation of $\Delta\bfT(A,E)$.
For each $k\in\calI$, we consider the $\DT(A,E)$-$D$ $\rmC^*$-correspondence $X^{(k)} = Y_k \otimes_{A_k} X_k \cong X$
with the left action $\sigma_k$ defined by the composition of the cut-off $\lL ( Y )\cong\prod_{i\in\calI}\lL(Y_i) \to \lL(Y_k)$
and the induced map $\lL ( Y_k ) \to \lL ( Y_k \otimes_{A_k} X_k )$.
Note that for each $k,j\in\calI$ with $k\neq j$ one has
\begin{align} \label{eq_sigma}
\sigma_k |_A  = \phi_X , \quad \sigma_k ( P_k ) = e_D + P_{(\ell,k)}^\perp, \quad \sigma_{j} ( P_k  ) = P_{(\ell,k)}^\perp,
\end{align}
and that $\bigoplus_k (X^{(k)}, \sigma_k)$ is faithful if $E_k$ is nondegenerate for every $k\in \calI$.

The next proposition follows from the definition, so we omit the proof.
\begin{proposition}\label{prop_proj}
The following hold true:
\begin{itemize}
  \item[(i)] The projections $P_k, k \in \calI$ commute with $D$.
  \item[(ii)] Any element $a \in A_k$ enjoys $P_k^\perp a  P_k^\perp =  E (a) P_k^\perp$ and $ a^\circ P_k^\perp = P_k^\circ a P_k^\perp$.
  \item[(iii)] The compression $\lL ( Y ) \to \lL ( \eta_k A_k ) \cong A_k$ by $e_{A_k}$ defines a conditional expectation from $\Delta\bfT(A,E)$ onto $A_k$ extending $E_{A_k}$. 
\end{itemize}
\end{proposition}
%
%
%
\section{Pimsner algebras}\label{section_pimsner}
\subsection{Extensions associated with conditional expectations}\label{ss_semisplit}
Let $D\subset A$ be a unital inclusion of $\rmC^*$-algebras with conditional expectation $E\colon A \to D$,
$(X,\phi_X,\xi_0)$ be the GNS representation associated with $E$,
and $e_D =\theta_{\xi_0,\xi_0} \in \lK(X)$ be the Jones projection onto $\xi_0D$.
Note that $e_D\phi_X(a)e_D=\phi_X(E(a))e_D$ holds for $a\in A$.
Let $\calB$ be the universal unital $\rmC^*$-algebra generated by a unital copy of $A$ and a projection $e$
such that $e ae =E(a)e$ for $a\in A$.

\begin{lemma}\label{lem_univ_rep}
A unital $*$-homomorphism $\rho$ from $\calB$ into a unital $\rmC^*$-algebra
is injective if the restrictions $\rho|_A$ and $\rho|_{De}$ are injective and
$\rho(A)\cap \ospan \rho(AeA)=\{0\}$.
\end{lemma}
\begin{proof}
Since $A+\lspan AeA$ is norm dense in $\calB$,
it suffices to show that $\|a + K\|\leq 3\|\rho(a+K)\|$ for all $a \in A$ and $K\in \lspan AeA$.
By assumption, we observe that $\|a\| \leq \|\rho (a+K) \|$
and the restriction of $\rho$ to $\ospan AeA$ is isometric (see, e.g. \cite[Proposition 4.6.3]{Brown_Ozawa}).
Thus, we have $\|a+K\|\leq \|\rho(a)\| +\|\rho(K)\| \leq \|\rho(a+K) \| + \|\rho (a+K) - \rho(a)\| \leq
2\|\rho(a+K)\| +\|\rho(a)\|\leq 3 \|\rho(a+K) \|$.
\end{proof}
By the lemma, we may identify $\calB$ with the $\rmC^*$-subalgebra of $A\oplus\lL(X)$ generated by $0\oplus e_D$ and $\{a \oplus \phi_X(a) \mid a\in A\}$
and set $B:=(1-e)\calB (1-e) \subset A\oplus \lL(X^\circ)$.
Under the identifications $\ospan AeA \cong \lK(X)$ and $\ospan A^\circ eA^\circ \cong \lK(X^\circ)$,
we have the following commuting diagram
\[
\xymatrix 
{
0 \ar[r] & \lK(X) \ar[r] & \calB \ar[r] & A \ar[r] & 0\\
0 \ar[r] & \lK(X^\circ) \ar[u] \ar[r] & B \ar [u] \ar[r] & A \ar@{=}[u] \ar[r] & 0
}
\]
such that the upper exact sequence is split
and the lower one is semisplit with the UCP cross section $\Psi\colon A \ni a \mapsto (1-e)a(1-e) \in B$.
We call $\calB$ and $(B,\Psi)$ the \emph{split} and \emph{semisplit extensions associated with} $(D\subset A,E)$, respectively.

Note that  every element in $B$ is of the form $\Psi (a)  + K  $ for some $a \in A$ and $K \in \lK (X^\circ )$.
For any $a\in \phi_X^{-1}(\lK(X))$, one has $e_D^\perp \phi_X(a)e_D^\perp \in \lK(X^\circ)$.
Thus, we have $\phi_X^{-1}(\lK(X))\cong \phi_X^{-1}(\lK(X))\oplus 0 \subset B$.
Note that this ideal is the kernel of the natural action of $B$ on $X^\circ$ via the surjection $A\oplus \lL(X^\circ) \to \lL(X^\circ)$.
The proof of the following lemma is easy, so we omit it.
\begin{lemma}\label{lem_Xcirc}
There exists an isometric linear map $t^\circ \colon X^\circ \to (1-e)\calB e$ such that
\begin{itemize}
\item $t^\circ (a \xi_0) =ae =(1-e)ae$ for $a\in A^\circ$;
\item $t^\circ (\xi)^*t^\circ(\eta)=\i<\xi,\eta>e$ for $\xi,\eta \in X^\circ$;
\item $t^\circ (b \xi d ) = b t^\circ(\xi) d$ for $b \in B, \xi \in X^\circ$ and $d \in D$. 
\end{itemize}
\end{lemma}

\begin{remark}\label{rem_group}
Assume that $(D\subset A, E)$ comes from the reduced group $\rmC^*$-algebras of discrete groups $\Lambda \leq \Gamma$.
Then, $[\Lambda:\Gamma]=\infty$ if and only if
$\phi_X^{-1}(\lK(X)) =\{0\}$.
Indeed, if $[\Lambda:\Gamma]<\infty$, then one has $A=\lK(X)$.
Note that this the case when $\calB = A \oplus \lK (X)$ and $B= A \oplus \lK(X^\circ)$.
Conversely, if $x \in \phi_X^{-1}(\lK(X))$ is nonzero,
then $x \otimes 1 \in \lK (X\otimes_D A)$ is also nonzero.
By the natural isomorphism $\lK (X\otimes_D A) \cong c_0 (\Gamma/\Lambda)\rtimes_\red \Gamma$,
we have $(1 \otimes \rmC^*_\red (\Gamma )) \cap (c_0 (\Gamma/\Lambda)\rtimes_\red \Gamma) \neq \{0 \}$.
This implies that $c_0 (\Gamma / \Lambda)$ is unital, so we have $[\Gamma :\Lambda] <\infty$.
Note that when $\phi_X(A)\cap\lK(X)= \{0\}$,
we have $\calB \cong \rmC^* ( \phi_X(A), e_D) \cong (\lC 1 +c_0 (\Gamma /\Lambda ) )\rtimes_\red \Gamma$.
\end{remark}

Consider the embedding maps $\lK(X^\circ) \hookrightarrow \lK (X)$ and $D \hookrightarrow De_D \subset \lK (X)$ into corners.
Recall that $D$ is said to have the \emph{completely bounded approximation property} (\emph{CBAP})
if there exist a constant $C>0$ and a net of finite rank CB maps $\varphi_i$ on $D$ such that $\| \varphi_i \|_\cb \leq C$
and $\lim_i \| \varphi_i (x) - x \| =0$ for $x\in D$.
The \emph{Haagerup constant} $\Lambda_\cb(D)$ is the infimum of those $C$ for which $(\varphi_i)_i$ exists.
When $D$ does not have the CBAP, we set $\Lambda_\cb(D)=\infty$.

\begin{lemma}\label{lem_cbap}
It follows that $\Lambda_\cb (\lK(X^\circ)) \leq \Lambda_\cb(D)$. 
\end{lemma}
\begin{proof}
Since $\lK(X^\circ)$ is a hereditary subalgebra of $\lK(X)$, we have $\Lambda_\cb(\lK(X^\circ))\leq \Lambda_\cb (\lK(X))$.
Take an approximate unit $(a_i \otimes p_i)$ of $\lK(X)\otimes\lK$.
It suffices to show that for any $i$ and any $\varepsilon>0$,
there exist CP contractions
$\varphi_i \colon \lK(X)\otimes \lK\to D\otimes\lK$ and $\psi_i\colon D\otimes\lK\to\lK(X)\otimes\lK$
such that $\| (a_i\otimes p_i)x(a_i\otimes p_i) -  \psi_i\circ\varphi_i(x) \| < \varepsilon \|x\|$ for
$x\in \lK(X)\otimes\lK$.
For each $i$,
by \cite[Lemma B.2]{Katsura}
we find a separable closed subspace $X_i\subset X$ with $\xi_0\in X_i$
which naturally forms a Hilbert $\rmC^*$-module over a separable $\rmC^*$-subalgebra $D_i \subset D$
such that $a_i \in \lK(X_i)$.
Since $D_i e_D \otimes \lK $ is a full corner of $\lK(X_i)\otimes\lK$,
it follows from \cite[Lemma 2.5]{Rordam} that
there exists $d_i \in \lK(X_i)\otimes\lK$ such that
$\| d_i^*d_i - a_i \otimes p_i \| <\varepsilon$ and $d_id_i^* \in De_D\otimes\lK$.
Then, CP contractions $\varphi_i(x)= d_i x d_i^*$ and $\psi_i (x)= d_i^*xd_i$ are the desired ones.
\end{proof}
\subsection{Cuntz--Pimsner algebras}
Let $(A, E ) = \bigast_D (A_k ,E_k )$ be the reduced amalgamated free product of $\{(D\subset A_k,E_k)\}_{k\in\calI}$.
Let $\calB_k \subset A_k\oplus \lK(X_k)$ and $B_k \subset A_k\oplus \lK(X_k^\circ)$
be the split and semisplit extension associated with $(D \subset A_k, E_k )$ as in \S\S\ref{ss_semisplit}.
For the UCP section $\Psi_k\colon A_k \to B_k$, we
consider the unital embedding $\Psi \colon D \to \prod_{k\in\calI}B_k; d \mapsto (\Psi_k(d))_{k\in\calI}$,
and the $\rmC^*$-algebra $B:= \bigoplus_{k\in\calI}B_k + \Psi (D)$.
We denote the support projection of $B_k$ in $B$ by $1_{B_k}$ and set $B_k^\perp = 1_{B_k}^\perp B$.
Define the Hilbert $B$-module $\frakX$ by $\bigoplus_{k\in \calI} X_k^\circ\otimes_D B_k^\perp$,
where the interior tensor product $X_k\otimes_D B_k^\perp$ is with respect to
$D \ni d \mapsto \Psi(d)1_{B_k}^\perp \in B_k^\perp$.
The left action $\phi_\frakX \colon B \to \lL (\frakX )$ is defined by the direct sum of the compositions
\[
B_k \hookrightarrow A_k\oplus \lL(X_k^\circ) \to \lL(X_k^\circ) \to \lL(X_k^\circ \otimes_D B_k^\perp).
\]
We set $\xi_{k\ol{k}}:=\xi_k \otimes 1_{B_k}^\perp\in\frakX$
and $I_k:=\phi_{X_k}^{-1}(\lK(X_k))$.
We may use the identifications $I_k \cong I_k\oplus 0 \subset B_k$ and
$\lK(X_k^\circ) \cong 0\oplus \lK(X_k^\circ) \subset B_k$.
Note that in the simplest case when $\calI=\{1,2\}$,
one has $B =B_1 \oplus B_2$ and $\frakX = (X_1^\circ \otimes_D B_2)\boxplus (X_2^\circ\otimes_D B_1)$.

\begin{lemma}\label{lem_J_X}
It follows that $\ker \phi_\frakX \supset \bigoplus_{k\in\calI}I_k$ and
$J_\frakX = \bigoplus_{k\in\calI} \lK (X^\circ_k )$.
\end{lemma}
\begin{proof}	
The first assertion is trivial.
Take $x \in J_\frakX$ arbitrarily.
Then there exist $a_k\in A_k$ and $K_k \in \lK(X_k^\circ)$ such that $x = (\Psi_k (a_k )+K_k)_{k\in\calI}$.
For each $k\in\calI$, it follows from Lemma \ref{lem_cpt_tensor} that $\phi_{X_k}(a_k ) \in \lK (X_k)$,
and so $a_k  \in I_k$.
Since $J_\frakX \subset (\ker \phi_\frakX)^\perp \subset (I_k)^\perp$ holds,
we obtain $(\Psi_k (a_k )+K_k)(a_k\oplus 0) =0$, implying $a_k =0$.
Since $k\in\calI$ is arbitrary,
we have $x= (K_k)_k \in \bigoplus_k \lK (X_k^\circ)$.
The opposite inclusion follows from that $\phi_\frakX ( \theta_{a \xi_k, b \xi_k } ) = \theta_{a \xi_{k\ol{k}}, b \xi_{k \ol{k}}}$ for $a, b \in A_k^\circ$ and $k \in \calI$.
\end{proof}

\begin{theorem}\label{prop_cuntz_pimsner}
There exists a universal covariant representation $(\pi, t )$ of $\frakX$ on $\Delta\bfT(A,E)$ such that $\pi ( \Psi_k (a)) = P_k a P_k$ for $a \in A_k$, $ t ( b \xi_{k \ol{k}} ) = b P_{k}^\perp$ for $b \in A_k^\circ$ and $k \in \calI$, and $\rmC^* ( \pi, t ) = \Delta\bfT(A,E)$.
\end{theorem}
\begin{proof}
For any $k\in\calI$, we have $e_{A_k}A_k(1-P_k)A_k e_{A_k}=\{0\}$.
Thus, by Lemma \ref{lem_univ_rep} and Proposition \ref{prop_proj}
there exists an injective $*$-homomorphism $\rho_k\colon \calB_k \to \DT(A,E)$ such that
$\rho_k|_{A_k}=\id$ and $\rho_k(e_k)=1-P_k$.
Hence, $\pi=\bigoplus_{k\in\calI}\rho_k|_{B_k}\colon B \to \DT(A,E)$ is an injective $*$-homomorphism.
Let $t^\circ_k \colon X_k^\circ \to \calB_k$ be as in Lemma \ref{lem_Xcirc}
and define $t \colon \frakX \to \DT(A,E)$ by $t(\xi \otimes x) =\rho_k(t_k^\circ(\xi)) \pi(x)$ for $\xi \in X_k^\circ$ and $x\in B_k^\perp$.
Then it follows from Lemma \ref{lem_Xcirc} and Lemma \ref{lem_J_X}
that $(\pi,t)$ is an injective covariant representation.
To see the universality of $(\pi,t)$, it suffices to show that $(\pi,t)$ admits a gauge action by Theorem \ref{thm_GIU}.
Indeed, for each $n \in \lN $, let $Q_n$ be the projection in $\lL ( Y)$ onto
the closed submodule generated by all vectors of the form $a_1 \cdots a_n \eta_k$
for some $k\in\calI$ and some reduced word $a_1 \cdots a_n$ with $a_n \notin A_k$,
and set $Q_0=\sum_{k\in\calI}e_{A_k}$.
Letting $U_z := \bigoplus_{n\geq 0} z^n Q_n$ for $z \in \lC$ with $|z|=1$
we have $\Ad U_z ( \pi (x)) = \pi (x)$
for $x \in B$ and $\Ad U_z ( t ( \xi )) = z t ( \xi )$ for $\xi \in \frakX$.
The inclusion $\Delta\bfT(A,E) \subset \rmC^* ( \pi,  t )$ follows from the decomposition
$a = P_k a P_k +P_k a^\circ P_k^\perp + P_k^\perp a ^\circ P_k + E (a) P_k^\perp$ for $a \in A_k$.
\end{proof}
Here are two corollaries.
The first one immediately follows from the proof of the previous theorem.
\begin{corollary}\label{cor_univ}
The $\rmC^*$-algebra $\DT(A,E)$ is the universal $\rmC^*$-algebra generated by a unital copy of the algebraic amalgamated
free product of $A_k,k\in\calI$ over $D$ and $c_0(\calI)$ with minimal projections $P_k,k\in\calI$ such that
for each $k\in\calI$, $e_k=1-P_k$ satisfies that $e_k a e_k =E_k(a)e_k$ for $a\in A_k$.
\end{corollary}

\begin{corollary}\label{cor_amalgam}
Let $(A,E)=(A_1,E_1)\ast_D(A_2,E_2)$ be a reduced amalgamated free product
and $\calB_k=\rmC^*(A_k,e_k)$ be the split extension associated with $(D\subset A_k,E_k)$ and $\pi_k\colon \calB_k \to A_k$ be the natural surjection.
Define the conditional expectation $\calE_k\colon \calB_k \to D(1-e_k) + De_k$ by
$\calE_k(x) = E_k(\pi_k(a))(1-e_k) +e_k x e_k$.
Then, $\DT(A,E)$ is the reduced amalgamated free product of $(\calB_k,\calE_k), k=1,2$
over $D\oplus D$, identified $e_1$ with $1-e_2$.
Moreover, the canonical surjection from the full amalgamated free product of $\calB_1$ and $\calB_2$ over $D\oplus D$ onto $\DT(A,E)$ is a $*$-isomorphism.
\end{corollary}
\begin{proof}
By the previous corollary, $\DT(A,E)$ is naturally isomorphic to the full amalgamated free product $\calB_1 \ast_{D\oplus D}\calB_2$.
Let $\sigma_k \colon \DT(A,E)\to \lL(X^{(k)})$ be as in \S\S \ref{ss_general}
and set $\theta:=\mathord{\id} \boxplus \sigma_1 \boxplus \sigma_2 \colon \DT(A,E) \to \lL(Y \boxplus X^{(1)} \boxplus X^{(2)})$.
Define the projection $f_1 \in \lL(Y \boxplus X^{(1)} \boxplus X^{(2)})$ by
$P_2 \oplus (e_D + P_{(\ell,2)}^\perp) \oplus P_{(\ell,2)}^\perp$
and put $f_2=1-f_1$.
For any $x \in \rmC^*(A_1,P_1) \cong \calB_1$ it follows from Eq.~(\ref{eq_sigma}) that
\begin{align*}
f_1 \theta(x)f_1 =
P_2xP_2 \oplus \sigma_1(P_1 x P_1)e_D + \sigma_1(P_2xP_2) P_{(\ell,2)}^\perp \oplus  \sigma_2(P_2 x P_2) P_{(\ell,2)}^\perp
=
\theta(\calE_1 (x))f_1.
\end{align*}
Similarly, one has $f_2\theta(y)f_2= \theta(\calE_2(y))f_2$ for $y \in \rmC^*(A_2,P_2) \cong \calB_2$.
Thus, by the previous corollary, the map $\calB_1 \ast_{D\oplus D} \calB_2 \to \DT(A,E)$ factors though the reduced one.
\end{proof}
%
%
We next show that the Toeplitz extension
\[
  0 \too \lK ( \calF ( \frakX) J_\frakX ) \too \calT (\frakX ) \too \calO (\frakX ) \too 0
\]
is semisplit.
Let $\sigma_k \colon \Delta\bfT(A,E) \to \lL (X^{(k)})$ be as in \S\S \ref{ss_general} and set $(X^\calI, \sigma ) := \bigboxplus_{k\in\calI} (X^{(k)},\sigma_k)$.
We denote the GNS vector in $X^{(k)}$ by $\xi_0^{(k)}$.
We fix a fixed-point free bijection $\tau$ on $\calI$.
To simplify the notation, we will write $\tau (k) = k+1$ for $k\in \calI$.
Let $Q \in \lL (X^\calI)$ be the projection onto $ \bigboxplus_{k\in\calI} P_{(r,k)}^\perp X^{(k+1)}$
and $\theta\colon \lL(X^\calI) \to \lL(QX^\calI)$ be the compression map.
Note that $P_{(r,k)}^\perp X^{(k+1)} \cong X(r,k)\otimes_D X_k^\circ$ contains a copy of
$X_k^\circ$, denoted by $X_k^{\circ(k+1)}$.
Since $QX^\calI$ is invariant under $\sigma\circ{\pi} ( B)$ and $\sigma\circ{ t }( \frakX )$,
the pair $(\pi', t' ) := (\sigma\circ{\pi} ( \cdot ) Q , \sigma\circ{t} ( \cdot ) Q )$ is
a representation of $\frakX$ on $\lL (Q X^\calI)$.

\begin{proposition}\label{prop_toeplitz}
The representation $(\Pi, T ) := ( \pi \oplus \pi' , t \oplus t' )$ of $\frakX$ is universal.
\end{proposition}
\begin{proof}
It is clear that $(\Pi, T )$ is injective and admits a gauge action.
Thus, we only have to check that $\Pi ( J_\frakX ) \cap \psi_{T} (\lK ( \frakX ) ) = \{ 0 \}$ by
Theorem \ref{thm_GIU}.
Assume that $x \in \lK ( X_k^\circ )$ satisfies $\Pi (x) \in \psi_{T} ( \lK ( \frakX ))$.
Observe that $t' ( \xi ) t' (\eta )^*$ vanishes on $\bigboxplus_{j\in\calI} X_j^{\circ(j+1)}$ for all $\xi, \eta \in \frakX$,
and hence so does $\pi' (x)$.
On the other hand, the restriction of $\pi' (x)$ to $X_k^{\circ(k+1)}$
is unitarily equivalent to $x$ itself on $X_k^\circ$.
Thus, $x$ must be zero.
\end{proof}

Since $(\Pi,T)$ is universal, there is a surjective $*$-homomorphism $p \colon \rmC^*(\Pi,T) \to \rmC^*(\pi,t)$
such that $p\circ \Pi =\pi$ and $p\circ T=t$.
Note that
the kernel of $p$ is generated by $  \{ \Pi( \theta_{a\xi_k,b\xi_k}) - \psi_T(\phi_\frakX(\theta_{a\xi_k,b\xi_k}))
\mid k\in\calI, a,b\in A_k^\circ \}$
by Lemma \ref{lem_J_X}.
\begin{theorem}\label{prop_semisplit}
The UCP map $\Theta := \id \oplus ( \theta \circ \sigma) \colon \Delta\bfT(A,E) \to  \Delta\bfT(A,E)  \oplus \lL ( QX^\calI )$
maps into $\rmC^* ( \Pi, T)$ and satisfies that $p \circ \Theta = \id$.
\end{theorem}
\begin{proof}
Since $DP_k$ sits in the multiplicative domain of $\Theta$
and $\Theta (c)=\Pi (\Psi_k(c)) + T(c^\circ \xi_{k\ol{k}}) + T(c^{*\circ}\xi_{k\ol{k}})^* + \Pi(\Psi( E(c))1_{B_k}^\perp) \in \rmC^*(\Pi, T)$
for $c \in A_k$ and $k\in\calI$,
it suffices to show that $\Theta (ab^*) - \Theta(a)\Theta(b^*) \in \ker p$
for $a \in A_k^\circ$, $b\in A_k^\circ$ and $k\in \calI$
and $\Theta (a_1 a_2 \cdots a_{n+1})= \Theta (a_1a_2 \cdots a_n)\Theta (a_{n+1})$ for
all reduced words $a_1\cdots a_n$ with $n\geq 1$.
Indeed, it follows from the above decomposition of $\Theta (c)$ that
 that
\begin{align*}
\Theta (ab^*) - \Theta(a)\Theta(b^*)
&= \Pi (\Psi_k(ab^*)) - \Pi(\Psi_k(a)\Psi_k(b^*)) - T(a\xi_{k\ol{k}})T(b\xi_{k\ol{k}})^* \\
&=\Pi( \theta_{a\xi_k,b\xi_k}) - \psi_T(\phi_\frakX(\theta_{a\xi_k,b\xi_k})),
\end{align*}
which belongs to $\ker p$.
We next show the multiplicativity on reduced words.
Take $k\in \calI$ and a reduced word $a_1 \cdots a_{n}$ with $a_{n} \in A_k^\circ$ arbitrarily.
Since $\pi(\Psi_k(A_k))$ and $t(\frakX)$ sit in the right multiplicative domain of $\Theta$,
we have
\begin{align*}
\Theta (a_1 \cdots a_n) - \Theta(a_1 \cdots a_{n-1})\Theta(a_n)
&= \Theta (a_1 \cdots a_{n-1}t(a_n^*\xi_{k\ol{k}})^* ) - \Theta(a_1 \cdots a_{n-1})T(a_n^*\xi_{k\ol{k}})^* \\
&= 0 \oplus Q  \sigma (a_1 \cdots a_{n-1}) (1 -Q) \sigma\circ{t} (a^*_n\xi_{k\ol{k}})^* Q.
\end{align*}
Note that $(1 -Q) \sigma\circ {t} (a^*_n\xi_{k\ol{k}})^* Q $ is supported on $X_k^{(k+1)\circ}$.
For any $x \in A_k^\circ$,
since $\sigma = \phi_X$ on $A$,
we have
$Q  \sigma (a_1 \cdots a_{n-1}) (1 -Q) \sigma\circ{t} (a^*_n\xi_{k\ol{k}})^* Q x \xi_0^{(k+1)}
 = Q  a_1 \cdots a_{n-1}\xi_0^{(k+1)} E(a_nx)
 =0$.
\end{proof}

\begin{remark}\label{rem_Theta}
The above proof shows that $\rmC^*(\Theta(A_k))$ is a semisplit extension of $A_k$ by $\lK(X_k^\circ)$.
The restriction of $E_\frakX(\cdot)1_{B_k} \colon \rmC^*(\Pi,T) \cong \calT(\frakX) \to B_k$ to $\rmC^*(\Theta(A_k))$
is a $*$-isomorphism onto $B_k$ such that $E_{\frakX}(\Theta(a))1_{B_k} =\Psi_k(a)$ for $a\in A_k$.
\end{remark}

\begin{corollary}\label{cor_app}
The $\rmC^*$-algebra $\Delta\bfT(A,E)$ is nuclear or exact if and only if $A_k$ has the same property for every $k\in \calI$.
Moreover, it follows that $\Lambda_\cb(\DT(A,E)) =\max\{\Lambda_\cb(A_k) \mid k \in \calI\}$.	
\end{corollary}
\begin{proof}
Assume that $A_k$ is nuclear (resp.~exact) for every $k\in \calI$.
Since $D$ is the range of the conditional expectation $E_1$, so is $D$.
By \cite[Proposition B.7]{Katsura} $\lK(X_k^\circ)$ has the same property,
hence so does the semisplit extension $B_k$.
When $\calI$ is finite, $B=\bigoplus_{k\in\calI} B_k$;
otherwise $B$ is a split extension of $D$ by $\bigoplus_k B_k$.
Therefore, $\calT (\frakX )$ is nuclear (resp.~exact) by \cite[Theorem 7.1, Theorem 7.2]{Katsura}.
Thus, the same holds for $\Delta\bfT(A,E)$ by Theorem \ref{prop_cuntz_pimsner} and Theorem \ref{prop_semisplit}.
The opposite implication follows from Proposition \ref{prop_proj} (iii).
The assertion for CBAP can be shown in the same manner together with \cite{Dykema_Smith} and Lemma \ref{lem_cbap}.
\end{proof}

%
%
\subsection{Boundary actions}\label{subsection_boundary}

Let $\Gamma=\Gamma_1 \ast_\Lambda \Gamma_2$ be
an amalgamated free product.
By the definition of the topology of $\DT$,
$\delta_{\Gamma_k}\in C(\DT)$ if and only if $[\Lambda:\Gamma_k]<\infty$.
Thus, $\partial\bfT$ is closed in $\DT$ if and only if $[\Lambda:\Gamma_k]<\infty$
for $k=1,2$.
In this case, the reduced crossed product $C(\partial\bfT)\rtimes \Gamma$ is the quotient
of $C(\DT)\rtimes_\red \Gamma$ by $c_0(\bfV)\rtimes_\red \Gamma$.
Let $(A,E)=(A_1,E_1)\ast_D(A_2,E_2)$ be the corresponding reduced amalgamated free product.
Via the unitary between $\ell^2(\Gamma/\Gamma_k)\otimes A$ and $Y_k\otimes_{A_k} A$
in the proof of Proposition \ref{prop_group} the projection $\delta_{\Gamma_k}\otimes 1$
corresponds to $e_{A_k}\otimes 1$,
and thus $c_0(\bfV)\rtimes_\red \Gamma$ corresponds to
$\lK(Y_1)\oplus \lK(Y_2)$,
generated by $e_{A_1}$ and $e_{A_2}$.
In this subsection, we see that $C(\partial\bfT)\rtimes_\red \Gamma$
is also a relative Cuntz--Pimsner algebra.

\medskip
Let $(A,E)=\bigast_D(A_k,E_k)$ be any reduced amalgamated free product.
Consider the ideals $J:=\phi_\frakX^{-1}(\lK(\frakX))$
and $J_0= \bigoplus_{k\in\calI} I_k \oplus\lK(X_k^\circ)$ of $B$.
Since $\phi_\frakX$ vanishes on $I_k \subset B_k$,
we have $J_0\subset J$.
\begin{lemma}
If either $\calI$ is finite, the left action $D \to \lL(X_k^\circ)$ is faithful for every $k\in\calI$,
or $I_k=\{0\}$ for every $k\in\calI$,
then $J=J_0$.
When $I_k=\{0\}$ for $k\in\calI$, we have $J=J_0=J_\frakX$.
\end{lemma}
\begin{proof}
Let $x=(\Psi_k(a_k+d)+K_k)_{k\in\calI} \in J$ be arbitrary
such that $d\in D$ and $(\Psi_k(a_k)+K_k)_k\in \bigoplus_kB_k$.
We observe that $a_k +d \in I_k$ for $k\in\calI$.
Thus, if $\calI$ is finite or $I_k=\{0\}$ for $k\in\calI$,
we have $J=J_0$.
Assume that $\calI$ is infinite and the left action $D\to\lL(X_k^\circ)$ is faithful for $k\in\calI$.
We then have $\|d\|= \inf_{k\in\calI} \|e_D^\perp \phi_{X_k}(d) \| = \|e_D^\perp\phi_{X_k}(a_k)e_D^\perp+K_k \| =0$.
Thus, we get $d=0$.
\end{proof}
For any ideal $J'\subset J$, the \emph{relative Cuntz--Pimsner algebra} $\calO(J',\frakX)$
is the quotient of $\calT(\frakX)$
by the ideal generated by $(\varphi_\infty-\psi_{\tau_\infty}\circ \phi_\frakX)(J')$
(c.f.~\cite{Muhly_Solel}).
We observe that  $\calO(J_0,\frakX)$
is isomorphic to the quotient of $\DT(A,E)$ by the ideal
$\bigoplus_{k\in\calI}\lK(Y_kI_k)$,
generated by
$(\pi-\psi_t\circ \phi_\frakX)(\bigoplus_{k\in\calI}I_k)=\ospan\{xe_{A_k} \mid x \in I_k, k\in\calI\}$.
Thus, if $I_k\neq \{0\}$ for some $k\in\calI$, then $\DT(A,E)$ is not simple.

In the group case as above, by Remark \ref{rem_group}, $[\Lambda:\Gamma_k] <\infty$ if and only if $I_k \neq \{0\}$ if and only if $I_k=A_k$.
Thus, when $[\Lambda:\Gamma_k]<\infty$ for $k=1,2$,
then we have $c_0(\bfV) \rtimes_\red \Gamma \cong \lK(Y_1)\oplus \lK(Y_2)$ and
$C(\partial\bfT)\rtimes_\red \Gamma \cong \calO(J, \frakX)$.

As in the group case,
one could expect that $\calO(J,\frakX)$ would be a more appropriate or smaller algebra
containing $A$ than $\calO(\frakX)$.
In this viewpoint, it is natural to ask
whether the quotient map from $\calO(\frakX) \to \calO(J,\frakX)$ is injective on $A$,
and the following gives a sufficient condition.
\begin{proposition}
If $E_k$ is nondegenerate and the left action $D \to \lL(X_k^\circ)$ is injective for all $k\in \calI$,
then the restriction of the quotient map $q\colon \DT(A,E) \to \partial \bfT(A,E):=\DT(A,E)/\bigoplus_{k\in\calI}\lK(Y_kI_k)$ to $A$ is injective.
\end{proposition}
\begin{proof}
It suffices to show that $\phi_{Y_k}(A)\cap \lK(Y_k)=\{0\}$ for $k\in \calI$.
Since $E_k$ is nondegenerate, this is equivalent to $\phi_X(A) \cap \rmC^* (\phi_X(A)P_{X_k}) =\{0\}$ for $k\in \calI$,
where $P_{X_k}$ is the projection onto $\xi_0D\oplus X_k^\circ$.
Fix an arbitrary element $a\in \phi_{X}(A)$ of norm one.
Then there exist $m\in\lN, \iota \in \calI_m$
and a unit vector $\zeta_1$ in $ X_{\iota(1)}^\circ \otimes_D \cdots \otimes_D X_{\iota(m)}^\circ$
such that $\| a \zeta_1\| > \delta > 0$.
We extend $\iota$ to $\lN$ in such a way that $\iota(j) \neq \iota (j+1)$ for $j\in\lN$.
To see that $a \notin \rmC^* (\phi_X(A)P_{X_k})$,
it is enough to find
a vector $\zeta_n$ in $X_{\iota(1)}^\circ \otimes_D \cdots \otimes_D X_{\iota(m+n)}^\circ$
such that $\| \zeta_n \| \leq 1$ and $\| a \zeta_n \| >\delta$ for every $n\in\lN$.
Assume that we have chosen $\zeta_j$'s up to $j=n$.
Since $\i<a \zeta_n, a \zeta_n >\in D$ satisfies
$ \delta < \|\i<a \zeta_n, a \zeta_n > \| \leq 1$
and $D$ acts on $X_{\iota(m+n+1)}^\circ$ faithfully,
there exists a unit vector $\zeta' \in X_{\iota(m+n+1)}^\circ$
such that $\|\i<{ \zeta', \i< a \zeta_n, a \zeta_n> \zeta'} > \| >\delta$.
Hence, $\zeta_{n+1} := \zeta_n \otimes \zeta'$ is the desired one.
\end{proof}
We close this section by the next proposition generalizing \cite[Theoerm 4.9]{Okayasu}.
Set $\calD:= \bigoplus_{k\in\calI}D1_{B_k} + \Psi(D)  \cong (\lC1 + c_0 (\calI)) \otimes D$. Define the Hilbert $\calD$-module $\frakY:=\bigoplus_{k\in \calI}X_k^\circ\otimes_{D}1_{B_k}^\perp \calD \subset \frakX$
and the left action $\phi_\frakY (d) = \phi_\frakX(d)|_\frakY$ for $d\in\calD$.
\begin{proposition}[{c.f.~\cite[Theorem 4.9]{Okayasu}}]
If $E_k$ is nondegenerate, the left action $D \to \lL (X_k^\circ)$ is injective,
and $I_k = A_k$ for every $k\in\calI$,
then $\partial\bfT(A,E)$ is isomorphic to $\calO(\frakY)$.
\end{proposition}
\begin{proof}
Let $(\bar{\pi},\bar{t})$ be the representation of $\frakY$ on $\dT(A,E)$ given by
$\bar{\pi} (d 1_{B_k}) = q(d P_k )$ and $\bar{t} (a\xi_{k\ol{k}})=q(a P_k^\perp)$ for
$k\in \calI$, $d \in D$ and $a \in B_k$.
Then $(\bar{\pi},\bar{t})$ is injective, covariant and admits a gauge action.
The surjectivity of the induced map $\rho\colon \calO(\frakY) \to \partial\bfT(A,E)$
follows from the decomposition
$ q(a)=  \psi_{\bar{t}} (\phi_\frakX (\Psi_k(a)) )  + \bar{t} (a \xi_{k\ol{k}}) + \bar{t} (a^*\xi_{k\ol{k}})^*$ for $a \in A_k^\circ$.
\end{proof}
\section{Applications}\renewcommand{\thetheorem}{\arabic{section}.\arabic{theorem}}\label{section_app}
In this section, we give proofs of several known results for reduced amalgamated free products
using $\DT(A,E)$.
The next theorem due to Dykema \cite{Dykema} (see \cite{Dykema_Shly,Ricard_Xu} for different proofs)
is an immediate consequence of Corollary \ref{cor_app}.
Note that our proof also says that any reduced amalgamated free product of nuclear $\rmC^*$-algebras
is a subalgebra of a nuclear $\rmC^*$-algebra.
Also, our result does not relay on the facts that exactness and nuclearity pass to quotients \cite{Choi_Effros,Kirchberg,Kirchberg_CAR}.
\begin{theorem}[Dykema \cite{Dykema}]
Reduced amalgamated free products of exact $\rmC^*$-algebras are exact.
\end{theorem}

The following generalizes Ozawa's result \cite{Ozawa_RIMS} for nuclearity.
\begin{theorem}
Let $(A,E)=(A_1,E_1)\ast_D(A_2,E_2)$ be a reduced amalgamated free product
and assume that the image of the GNS representation
of $E_1$ contains the Jones projection
and $E_2$ is nondegenerate.
Then, $A$ is nuclear if and only if so are $A_1$ and $A_2$.
Also, one has $\Lambda_\cb(A)=\max\{\Lambda_\cb(A_1),\Lambda_\cb(A_2)\}$.
\end{theorem}
\begin{proof}
Let $p\in A_1$ be a projection such that $\phi_{X_1}(p)$ is the Jones projection of $E_1$.
A direct computation shows that $p=pe_{A_1} +P_2$ in $\DT(A,E)$.
Set $I_1:=\phi_{X_1}^{-1}(\lK(X_1))$.
Then, $\lK(Y_1I_1)$ is the ideal of $\DT(A,E)$ generated by $I_1e_{A_1}$
and one has $\lK(Y_1I_1) \cap A = \{0\}$.
Since $P_2 \in A +\lK(Y_1I_1)$,
we have $\DT(A,E)=A+\lK(Y_1I_1)$ and the split exact sequence
\[
0 \too \lK(Y_1I_1) \too A + \lK(Y_1I_1) \too A \too 0.
\]
Thus, the assertion follows from Corollary \ref{cor_app}.
\end{proof}

\begin{theorem}[Blanchard--Dykema \cite{Blanchard_Dykema}]
Let $(A,E)=\bigast_D (A_k,E_k)$ and $(\calA,\calE)=\bigast_D (\calA_k,\calE_k)$ be 
reduced amalgamated free products.
\begin{itemize}
\item[(i)] For any unital $*$-homomorphisms $\pi_k \colon A_k \to \calA_k$, $k\in\calI$ such that $\calE_k\circ \pi_k=\pi_k\circ E_k$
there exists a unique $*$-homomorphism $\pi \colon A \to \calA$ of which restriction to $A_k$ equals to $\pi_k$ for every $k\in\calI$.
Moreover, if $\pi_k$ is injective for every $k\in \calI$, then so is $\pi$.
\item[(ii)] Assume that $D=\calD$ and $E_k$ and $\calE_k$ are nondegenerate
for $k\in\calI$.
For any UCP maps $\varphi_k\colon A_k \to \calA_k$ such that $\varphi_k|_D =\id$,
there exists a unique UCP map $\varphi \colon A \to \calA$ such that
\[
 \varphi(a_1a_2\cdots a_n)=\varphi_{\iota(1)}(a_1)\varphi_{\iota(2)}(a_2) \cdots \varphi_{\iota(n)}(a_n)
\]
for any reduced word $a_1 \cdots a_n$ with $a_j\in A_{\iota(j)}^\circ$ and $\iota \in \calI_n$.
\end{itemize}
\end{theorem}
\begin{proof}
We prove (i):
By Corollary \ref{cor_univ} the $*$-homomorphisms $\pi_k,k\in\calI$ induce $\rho \colon \DT(A,E) \to \DT(\calA,\calE)$.
The covariant representation of $\frakX$ corresponding to $\rho$ admits a gauge action
since $\DT(\calA,\calE)$ admits a gauge action,
and is injective  whenever $\pi_k$ is injective for $k\in\calI$.
The restriction of $\rho$ to $A$ is the desired one.

We prove (ii):
The argument here is essentially same as the proof of \cite[Proposition 2.1]{Choda}.
Let $(\calX_k,\phi_{\calX_k},\xi_k')$ and $(\calX,\phi_\calX,\xi_0')$
be the GNS representation of $\calE_k$ and $\calE$.
For any $k\in\calI$,
by the Stinespring construction, there exists an $A_k$-$D$ $\rmC^*$-correspondence
$(Z_k,\pi_k)$ and an isometry $w_k\colon \calX_k \to Z_k$ such that
$w_k^*\pi_k(\phi_{X_k}(a))w_k =\phi_{\calX_k}(\varphi_k(a))$ for $a\in A_k$.
Let $E_k'\colon \lL(Z_k) \to D$ be the conditional expectation
given by the compression onto $w_k\xi_k'$.
Let $(L,E')$ be their reduced amalgamated free product
and $X'$ be the GNS Hilbert $\rmC^*$-module associated with $E'$.
Let $w\colon \calX \to X'$ be the isometry defined by
$w (\zeta_1 \otimes \cdots \otimes \zeta_n) =w_{\iota(1)}\zeta_1 \otimes \cdots \otimes w_{\iota(n)}\zeta_n$ for $\iota \in \calI_n$ and $\zeta_k \in X^\circ_{\iota(k)}$ for $k=1,2,\dots,n$.
By (i) we have a $*$-homomorphism $\pi \colon A \to L$ induced from $\pi_k$'s.
Then, the UCP $A \ni a \mapsto w^*\pi (a)w \in \calA$ is the desired one.
\end{proof}

%
%
\section{$KK$-theory}\label{section_K}
In this section, we give a new proof of the next theorem due to Fima and Germain \cite{Fima_Germain1}.
Key ingredients of our proof are the right invertibility of the embedding $A \hookrightarrow \DT(A,E)$ in $KK$-theory and exact sequences of $KK$-groups for Cuntz--Pimsner algebras
by Pimsner \cite{Pimsner_free}.
\begin{theorem}[Fima--Germain]\label{thm_FG}
	Let $(A, E) = (A_1 , E_1 ) \ast_D (A_2, E_2 )$ be a reduced amalgamated free product of
	unital separable $\rmC^*$-algebras
	and $i_k \colon D \hookrightarrow A_k$ and $j_k \colon A_k \hookrightarrow A$
	be inclusion maps for $k=1,2$.
	Then, there are two cyclic exact sequences for any separable $\rmC^*$-algebra $P$:
	\[
	\begin{CD}
	KK(P,D) @>(i_{1*},i_{2*}) >> KK(P,A_1) \oplus KK (P, A_2) @>j_{1*}- j_{2*}>>KK( P, A )  \\
	@AAA                                    @.                                       @VVV \\
	KK^1(P,A  ) @<j_{1*}- j_{2*}<< KK^1(P,A_1 )\oplus KK^1 (P, A_2) @<(i_{1*},i_{2*}) <<KK^1( P, D)
	\end{CD}
	\]
	and
	\[ \begin{CD}
	KK(D,P) @<i_1^*- i_2^*<< KK(A_1,P) \oplus KK (A_2,P) @<(j_1^*, j_2^* )<< KK( A, P)  \\
	@VVV                                    @.                                       @AAA \\
	KK^1(A , P ) @>(j_1^*, j_2^* )>> KK^1 (A_1, P )\oplus KK^1 (A_2, P) @>i_1^*- i_2^*>> KK^1 ( D,P).
	\end{CD}
	\]
\end{theorem}
Let $\phi =\phi_Y \colon A \hookrightarrow \Delta\bfT(A,E)$ be the inclusion map and set $\rho_k :=\pi \circ\Psi_k |_D \colon D \hookrightarrow D P_k \subset \Delta\bfT(A,E)$.
For each $k\in\{1,2\}$, we denote by $\ol{k}$ the unique element in $\{1,2\} \setminus \{k\}$.
We first show the right invertibility of $\phi \in KK(A,\DT(A,E))$.
\begin{lemma}\label{lem_subequiv}
There exist $\alpha \in KK ( \Delta\bfT(A,E), A)$ and $\delta \in KK ( \Delta\bfT(A,E), D)$ such that
$(\phi\oplus \rho_1) \otimes_{\DT(A,E)} (\alpha \oplus \beta ) = \id_A \oplus \id_D$. 
\end{lemma}
\begin{proof}
Let $(Z, \phi_Z)$ and $\widetilde{\phi}_Z \colon \Delta\bfT(A,E) \to \lL (Z)$ be as in Eq.\,(\ref{eq_YZ}).
Define the isometry $S \colon X \otimes_D A \to Z$ by
\[
S = 
\begin{cases}
  S_1 \otimes 1 :  X (r, 1) \otimes_D A \to Y_1 \otimes_{A_1} A ; \\
  S_2 \otimes 1 :  X (r, 2 )^\circ \otimes_D A \to Y_2^\circ \otimes_{A_2} A.
\end{cases}
\]
It follows from \cite[Lemma 3.2]{Hasegawa_IMRN} that $S (\phi_X ( a) \otimes 1 ) - \phi_Z (a ) S $ is compact for $a\in \Delta\bfT(A,E)$.
Since $\widetilde{\phi}_Z ( P_1 ) S = S ( \sigma_1 (P_1 ) \otimes 1 ) $ holds,
the triplet
\[
  (  Z \oplus ( X \otimes_D A ),  \widetilde{\phi}_Z \oplus ( \sigma_1 \otimes 1 ), \left[ \begin{smallmatrix} 0 & S \\ S^* & 0 \end{smallmatrix}\right]  )
\]
is a $\Delta\bfT(A,E)$-$A$ Kasparov bimodule
and defines an element $\alpha \in KK ( \Delta\bfT(A,E), A)$.
Since $\phi  \otimes_{\DT (A,E)} \alpha $ is implemented by the $A$-$A$ Kasparov bimodule
\[
  (Z \oplus (X \otimes_D A ),  \phi_Z \oplus (\phi_X \otimes 1) , \left[ \begin{smallmatrix} 0 & S \\ S^* & 0 \end{smallmatrix}\right]  ) ,
\]
we have $\phi \otimes_ {\Delta\bfT(A,E)} \alpha = \id_A  \in KK (A, A)$ by \cite[Theorem 3.4]{Hasegawa_IMRN} (see \cite{Fima_Germain1} for the degenerate case).
Also, it follows from $\widetilde{\phi}_Z ( P_1 ) = S ( \sigma_1 (P_1 ) \otimes 1 )S^*$ that $\rho_1  \otimes_{\DT (A,E)} \alpha  = 0$.
Let $\sigma_k \colon \Delta\bfT(A,E) \to \lL(X^{(k)})$ be as in \S\S \ref{ss_general}.
Since $\sigma_1 = \sigma_2 =\phi_X$ on $A$ and $\sigma_1 ( P_1 ) - \sigma_{2} ( P_1 ) = e_D$ hold,
the triplet
$(X^{(1)} \oplus X^{(2)}, \sigma_1 \oplus \sigma_2, \left[ \begin{smallmatrix} 0 & 1 \\ 1 & 0 \end{smallmatrix}\right] )$
is a $\Delta\bfT(A,E)$-$D$ Kasparov bimodule and the corresponding element $\delta \in KK ( \Delta\bfT(A,E), D)$ satisfies that $\rho_1 \otimes_{\DT (A,E)} \delta  = \id_D$ and $\phi \otimes_{\DT (A,E)} \delta = 0$.
\end{proof}

To compute the $K$-theory of $\DT(A,E) \cong \calO(\frakX)$,
we need to compute the $K$-theory of $J_\frakX$.
Thanks to the next technical lemma,
we can assume that $X_1^\circ$ and $X_2^\circ$ are \emph{full},
i.e., $\ospan \{E(a^*b) \mid a,b\in A_k^\circ \} =D$ holds for each $k=1,2$.
Note that $X_k^\circ$ is full whenever $D$ is simple (e.g. $D=\lC$).

\begin{lemma}\label{lem_nonfull}
Let $\varphi$ be a nondegenerate state on $D$ and set $\varphi_k : = \varphi \circ E_k$ for $k =1,2$,
$(\calT , \omega )$ be the Toeplitz algebra with the vacuum state,
and $(\calA_k, \widetilde{\varphi}_k ) = ( A_k , \varphi_k ) \ast ( \calT , \omega )$ be the reduced free product.
Denote by $F_k \colon \calA_k \to D$
the composition of the canonical conditional expectation $\calA_k \to A_k$ and $E_k \colon A_k \to D$
and by $\calX_k$ the GNS Hilbert $\rmC^*$-module of $F_k$.
Set $(\calA, F ) = (\calA_1, F_1 ) \ast_D (\calA_2 , F_2 )$.
Then $\calX_k^\circ$ is full and the embedding maps $A_k \hookrightarrow \calA_k$ and $A \hookrightarrow \calA$ induce $KK$-equivalences.
\end{lemma}
\begin{proof}
Let $s$ be the unilateral shift generating $\calT$.
Then, one has $F_k (s)= 0$ and $F_k(s^*s) =1$, and thus $\calX_k^\circ$ is full.
Let $\calH_k$ and $\calH$ be the $\rmC^*$-correspondences over $A_k$ and $A$
associated with the UCP maps $\varphi_k(\cdot) 1$ and $\varphi(\cdot)1$ on $A_k$ and $A$ (see e.g. \cite[Example 4.6.11]{Brown_Ozawa}), respectively.
It follows from \cite[Theorem 2.3]{Shlyakhtenko}
that $(\calT (\calH_k ), \varphi_k \circ E_{\calH_k } ) \cong (A_k, \varphi_k ) \ast (\calT, \omega )$.
Thus, the embedding $A_k \hookrightarrow \calA_k$ induces a $KK$-equivalence by \cite{Pimsner_free}.
Similarly, by a result due to Speicher \cite{Speicher}, we have
\[
(\calA,\varphi \circ F) \cong (A, \varphi \circ E ) \ast
(\calT, \omega ) \ast (\calT, \omega ) \cong (\calT(\calH), E_\calH) \ast_A (\calT(\calH),E_\calH)
\cong (\calT(\calH\oplus \calH), E_{\calH\oplus\calH}),
 \]
and thus $A \hookrightarrow \calA$ gives a $KK$-equivalence by \cite{Pimsner_free} again.
\end{proof}

We may identify $\calB_k$ with $\rmC^*(A_k, P_k)$.
If $X_1^\circ$ and $X_2^\circ$ are full,
then $\lK(X_k^\circ) \subset \lK(X_k)$ and $B_k \subset \calB_k$ are full corners,
and thus all the horizontal embedding maps in the next commuting diagram
induce $KK$-equivalences by \cite{Brown}.
\begin{equation}\label{diag_cpt}
\begin{gathered}
\xymatrix @C=1.2cm@R=1.2cm
{
\lK(X_k^\circ)  \ar@{^{(}->}[r]^{\kappa_k} \ar@{^{(}->}[d]
    & \lK(X_k)\ar@{^{(}->}[d]
      & D \ar@{_{(}->} [l]_{\epsilon_k} \ar[ld]^{\rho_{\ol{k}}} \\
B_k \ar@{^{(}->}[r] & \calB_k &
}
\end{gathered}
\end{equation}
Here the embedding $\epsilon_k \colon D \hookrightarrow \lK(X_k)$ is given by $d\mapsto d e_D$.
We set $(X_k^\circ):= \kappa_k\otimes_{\lK(X_k)} (\epsilon_k)^{-1} \in KK (\lK(X_k^\circ), D)$.
Since the inclusion map $\mu_k \colon A_k \hookrightarrow \calB_k$ gives
a cross section of of the split extension $0 \too \lK(X_k) \too \calB_k \too A_k \too 0$ (see \S\S\ref{ss_semisplit}),
the above commuting diagram implies that
$\mu_k \oplus \rho_{\ol{k}} \in KK(A_k\oplus D , \calB_k)$ is a $KK$-equivalence.

By the results in \S\ref{section_pimsner}
we may use the identifications $\DT(A,E) =\calO(\frakX)$ and $\rmC^*(\Pi,T)=\calT(\frakX)$.
Recall that the kernel of the quotient map
$p \colon \calT(\frakX) \to \calO(\frakX)$ is isomorphic to $\lK(\calF(\frakX)J_\frakX)$,
and that $\Pi \colon B \to \calT(\frakX)$ and
$\iota_\Omega :=(\Pi - \psi_T\circ\phi_\frakX)|_{J_\frakX} \colon J_\frakX \to \lK(\calF(\frakX)J_\frakX)$
induce $KK$-equivalences (see \cite{Pimsner_free,Katsura}).
By these $KK$-equivalences,
one can rewrite the six-term exact sequence induced from the Toeplitz extension
as
\begin{equation}
\begin{CD}\label{seq_CP}
KK ( P,   J_\frakX  )  @>>{\iota_* - [ \frakX]}> KK( P, B  )  @>>{\pi_*}>  KK (P, \calO(\frakX)) \\
@AAA				 	@.					@VVV\\
KK^1 (P,  \calO(\frakX) )  @ <{\pi_*}<< KK^1 (P, B )  	@<{\iota_*-[\frakX]}<<  KK^1 ( P, J_\frakX)
\end{CD}
\end{equation}
for any separable $\rmC^*$-algebra $P$ (see \cite{Pimsner_free,Katsura}).
Here $\iota$ is the inclusion map $J_\frakX \hookrightarrow B$ and
$[ \frakX ]$ is induced from $[( \frakX, \phi_\frakX |_{J_\frakX } , 0 )] \in KK (J_\frakX, B)$.
Let $\delta_p \in KK^1 (\calO(\frakX), \lK (\calF (\frakX)J_\frakX ))$ be the element corresponding
to the Toeplitz extension (see \cite[\S 1]{Skandalis}).
Then, the connecting map $KK (P, \calO(\frakX)) \to KK^{1} (P, J_\frakX)$ in the above exact sequence is given by
$ \delta_p \otimes_{\lK(\calF(\frakX)J_\frakX)} (\iota_\Omega)^{-1} \in KK^1 (\calO (\frakX), J_\frakX)$.
Note that $(X_1^\circ) \oplus (X_2^\circ) \in KK(J_\frakX, D\oplus D)$ is a $KK$-equivalence.

\begin{lemma}\label{lem_xieta}
Assume that $X_1^\circ$ and $X_2^\circ$ are full.
Then there is a cyclic exact sequence
\[
\begin{CD}
KK(P,D\oplus D) @>>\xi> KK (P, A_1  \oplus A_2 \oplus D  \oplus D) @>>\eta>  KK(P, \calO(\frakX)) \\
@AA\partial A			@.						@V\partial VV\\
KK^1 (P, \calO(\frakX))  @ <\eta<< KK^1 (P,A_1  \oplus A_2 \oplus D  \oplus D)	@<\xi<<  KK^1 (P, D  \oplus D),
\end{CD}
\]
where $ \xi (x, y ) = ( - i_{1 *} (y ),-i_{2*}(x), x  +y , x + y )$ and
$\eta = \phi_* \circ j_{1*} + \phi_* \circ j_{2*} + \rho_{1*} + \rho_{2*}$.
The map $\partial$ is induced from $ \delta_p \otimes_{\lK(\calF(\frakX)J_\frakX)}
(\iota_\Omega)^{-1} \otimes_{J_\frakX} ((X_1^\circ)\oplus (X_2^\circ) ) \in KK^1 (\calO(\frakX), D\oplus D)$.
\end{lemma}
\begin{proof}
The proof proceeds by rewriting the exact sequence in Eq.~(\ref{seq_CP}).
Since $X_k^\circ$ is full, $B_k$ is a full corner of $\calB_k$,
and thus the inclusion $B_k\hookrightarrow \calB_k$ induces a $KK$-equivalence.
Then, the next diagram commutes and all vertical arrows are isomorphisms:
\[
\xymatrix@R=1.2cm @C=2.2cm{
KK^p(P, \lK (X_k^\circ ) ) \ar[d]_{(\kappa_k)_*}\ar[r]_{\iota_*} & KK^p(P, B_k )\ar[d] \ar[r]_{\pi_*} & KK^p(P,\calO(\frakX)) \\
KK^p(P, \lK(X_k)) \ar[r] & KK^p(P,\calB_k) \ar[r] & KK^p(P,\calO(\frakX)) \ar@{=}[u]\ar@{=}[d]\\
KK^p(P, D ) \ar[u]^{(\epsilon_k)_*} \ar[r]^{(0,1)} &KK^p(P, A_k \oplus D )\ar[u]_{(\mu_k)_* +(\rho_{\ol{k}})_*} \ar[r]^{\phi_*\circ j_{k*}+\rho_{\ol{k}*} }& KK^p(P,\calO(\frakX)).
}
\]
We next observe that $[ \frakX ]$ is the direct sum of two maps
$(\Psi_{\ol{k}}\circ i_{\ol{k}})_* \circ (X_k^\circ)_* \colon
KK^p(P, \lK ( X_k^\circ )) \to KK^p(P, B_{\ol{k}})$, $k=1,2$.
Thus, the assertion follows from next commuting diagram:
\begin{equation}\label{seq_X_k}
\begin{gathered}
\xymatrix @R=1.4cm @C=3cm
{
KK^p(P, \lK ( X_k^\circ ) ) \ar[d]^{(X_k^\circ)_*}\ar[r]^{(\Psi_{\ol{k}}\circ i_{\ol{k}})_* \circ (X_k^\circ)_* } & KK^p(P, B_{\ol{k}})\ar[d]\\
KK^p(P,D )\ar[ur]_{(\Psi_{\ol{k}}\circ i_{\ol{k}})_*}\ar[r]_{(i_{\bar{k}*} ,  \ - (\cdot) )} & KK^p(P,A_{\ol{k}}\oplus  D ).
}
\end{gathered}
\end{equation}
\end{proof}
We are now ready to prove Theorem \ref{thm_FG}.
\begin{proof}[Proof of Theorem \ref{thm_FG}]
We first show the exactness at $KK^p(P,A_1\oplus A_2)$.
By the previous lemma,
we have $\Imag \partial \circ \phi_*  \subset \ker \xi = \{(x, -x) \mid x \in \ker (i_{1*},i_{2*})\}$.
Via the isomorphism $\ker \xi \ni (x,-x) \mapsto x \in \ker (i_{1*}, i_{2*}) \subset KK^p(P,D)$,
we obtain a connecting map $\partial' \colon KK^p(P, A) \to KK^{p+1} (P, D)$.
By Lemma \ref{lem_subequiv}, $\phi_*\colon KK^p(P,A)\to KK^p(P,\calO(\frakX))$ is injective.
Thus, $j_{1*}(x)-j_{2*}(y)=0$ if and only if
$\eta(x,-y,0,0)=0$ if and only if
$(x, -y, 0, 0 ) =( i_{1*}(z), -i_{2*}(z))$ for some $z \in KK^p(P,D)$,
and hence we obtain $\Imag (i_{1*}, i_{2*}) = \ker (j_{1*} - j_{2*})$.
Also, since $\phi_*$ is injective,
we have $\ker \partial' = \ker \partial \circ \phi_* = \Imag (j_{1*}+ j_{2*})$,
and thus the exactness at $KK^p(P,A)$ holds.

To see $\Imag \partial' = \ker (i_{1*}, i_{2*})$, it is enough to see $\Imag \partial \circ \phi_*  \supset \ker \xi = \{ (x, -x ) \mid x \in \ker(i_{1*}, i_{2*}) \} $.
By the definition of $\partial$,
this is equivalent to $\Imag (\phi \otimes_{\calO(\frakX)} \delta_p)_*$ contains
$(\iota_\Omega)_* \circ ((X_1^\circ)\oplus (X_2^\circ) )^{-1}_*(x, -x )$ for $x\in \ker (i_{1*},i_{2*})$.
Let $\Theta \colon \calO(\frakX) \to \calT(\frakX)$ be as in Theorem \ref{prop_semisplit}
and put $\frakA := \rmC^*(\Theta(A)) + \lK(\calF(\frakX)J_\frakX)$.
The next commuting diagram
\[
\xymatrix @R=1.2cm
{
0 \ar[r] & \lK(\calF(\frakX)J_\frakX) \ar[r] & \calT(\frakX) \ar[r]_p & \calO(\frakX) \ar[r] & 0\\
0 \ar[r] & \lK(\calF(\frakX)J_\frakX) \ar@{=} [u]\ar[r]^{\nu} & \frakA\ar @{^{(}->}[u] \ar[r]^q & A \ar[u]_{\phi} \ar[r] & 0
}
\]
and \cite[Lemma 1.5]{Skandalis} imply that $\phi \otimes_{\calO(\frakX)} \delta_p \in KK^1(A,\lK(\calF(\frakX)J_\frakX))$ is
the element corresponding to the semisplit extension $\frakA$ of $A$.
Hence it follows from the six-term exact sequence for $\frakA$ that
$\Imag (\phi \otimes_{\calO(\frakX)} \delta_p)_* = \ker \nu_*$.
Therefore, it suffices to show that
$\nu_* \circ (\iota_\Omega)_* \circ ((X_1^\circ)\oplus (X_2^\circ) )^{-1}_*(x, -x )=0$ for $x \in \ker(i_{1*}, i_{2*})$. 
Let $\theta_k \colon B_k \to \rmC^* (\Theta_k(A_k)) $ be the inverse of the isomorphism  in Remark \ref{rem_Theta}.
Note that $\iota_\Omega \otimes_{\lK(\calF(\frakX)J_\frakX)} \nu= \iota \otimes_{B} (\theta_1 + \theta_2)$ in $KK (J_\frakX, \frakA)$.
Also, it follows from the proof of the previous lemma that
$((X_1^\circ)\oplus (X_2^\circ) )^{-1}_* (x,-x) \in((X_1^\circ)\oplus (X_2^\circ) )^{-1}_* (\ker \xi) =\ker (\iota_* - [\frakX])$.
Since $\theta_1 \circ \Psi_1 \circ i_1 = \theta_2 \circ \Psi_2 \circ i_2$,
we have
\begin{align*}
\nu_* \circ (\iota_\Omega)_* \circ ((X_1^\circ)\oplus (X_2^\circ) )^{-1}_*(x, -x )
&=  (\theta_1 + \theta_2)_* \circ \iota_* \circ ((X_1^\circ)\oplus (X_2^\circ) )^{-1}_*(x, -x )\\
&=  (\theta_1 + \theta_2)_* \circ ( -[\frakX]) \circ ((X_1^\circ)\oplus (X_2^\circ) )^{-1}_*(x, -x )\\
&=  (\theta_1 + \theta_2)_* \circ (\Psi_{1} \circ i_1,\Psi_2 \circ i_2)_*  (-x, x )\\
&=  - (\theta_1 \circ \Psi_{1} \circ i_1)_* (x) + ( \theta_2 \circ \Psi_2 \circ i_2)_*  (x )\\
&=  0.
\end{align*}
Here the third equality follows from the commuting diagram (\ref{seq_X_k}).
Thus, we obtain the exact sequence for $KK^p(P,- )$.

For the exact sequence for $KK^p (- P)$,
it is enough to show that $\phi \oplus \rho_1$ is a $KK$-equivalnce.
Let $\eta$ be as in Lemma \ref{lem_xieta}.
Since $\rho_1 +\rho_2 = j_1\circ i_1$ holds in $KK(D, \calO(\frakX))$,
a simple diagram chasing shows that $\phi_* + \rho_{1*} \colon KK (\calO(\frakX), A \oplus D) \to KK(\calO(\frakX), \calO(\frakX))$
is surjective.
We show that $\phi +\rho_1$ is a $KK$-equivalence by the following trick from \cite{Pimsner_tree}:
Take $\gamma \in KK ( \calO(\frakX), A \oplus D)$ such that $ 1_{\calO(\frakX)} - (\alpha \oplus \delta)\otimes_{A \oplus D}(\phi \oplus\rho_1) =  \gamma \otimes_{A\oplus D}(\phi\oplus\rho_1)$.
Since the left hand side is an idempotent in the ring $KK(\calO(\frakX),\calO(\frakX))$,
it follows from Lemma \ref{lem_subequiv} that
\begin{align*}
\gamma \otimes_{A\oplus D}(\phi\oplus\rho_1)
&= \gamma \otimes_{A\oplus D}(\phi\oplus\rho_1) - \gamma \otimes_{A\oplus D}(\phi\oplus\rho_1) \otimes_{\calO(\frakX)}(\alpha \oplus \delta)\otimes_{A \oplus D}(\phi \oplus\rho_1) \\
&= \gamma \otimes_{A\oplus D}(\phi\oplus\rho_1)-\gamma \otimes_{A\oplus D}(\phi\oplus\rho_1)=0.
\end{align*}
\end{proof}

\begin{remark}\label{rem_connect}
Our proof shows that if $X_1^\circ$ and $X_2^\circ$ are full,
then the composition of the connecting map $\partial' \colon KK(P,A) \to KK^1(P,D)$,
the diagonal embedding $D \to D\oplus D$, and the $KK$-equivalence $(X_1^\circ\oplus X_2^\circ)^{-1}\otimes_{J_\frakX} \iota_\Omega \in KK(D\oplus D, \lK( \calF(\frakX)J_\frakX))$ is given by
the semisplit extension
\[
 0 \too \lK(\calF(\frakX)J_\frakX) \too \rmC^* (\Theta (A)) + \lK(\calF(\frakX)J_\frakX) \too A \too 0.
\]
In the original proof in \cite{Fima_Germain1},
it was shown that a natural embedding of the mapping cone $C_i$ of the diagonal embedding $D\to A_1 \oplus A_2$ into the suspension $SA$ has an inverse $x \in KK(SA,C_i)$.
Then, the connecting map is given by the Kasparov product of $x$ and the evaluation map $C_i \to D$.
It might be interesting to compare these two maps.
 \end{remark}

As a by-product, we obtain the following theorem:
\begin{theorem}
The element $\phi \oplus \rho_1 \colon KK(A\oplus D, \DT(A,E))$ is a $KK$-equivalence.
Therefore, the $KK$-class of $\DT(A,E)$ does not depend on the choice of conditional expectations $E_k, k=1,2$.
\end{theorem}
\begin{proof}
When $X_1^\circ$ and $X_2^\circ$ are full,
the assertion follows from the last paragraph of the proof of Theorem \ref{thm_FG}.
In the general case,
one can check the surjectivity of $\phi_* +\rho_{1*}$
by applying Theorem \ref{thm_FG} to the reduced amalgamated free products $A$ and $\DT(A,E)$ by Corollary \ref{cor_amalgam}.
The second assertion follows from the $KK$-equivalence between $A$ and
the corresponding full amalgamated free product (see \cite{Hasegawa_IMRN,Fima_Germain1}).
\end{proof}

We close this paper by the following corollary about $K$-nuclearity introduced by Skandalis \cite{Skandalis2}.
Note that this corollary also follows from the original proof in \cite{Fima_Germain1}
since the mapping cone of the diagonal embedding $i\colon D \to A_1 \oplus A_2$
is a semisplit extension of $D$ by $SA_1\oplus SA_2$.

\begin{corollary}\label{cor_K_nuc}
Reduced amalgamated free products of $K$-nuclear $\rmC^*$-algebras are $K$-nuclear.
\end{corollary}
\begin{proof}
Assume that $A_1, A_2$ and $D$ are $K$-nuclear.
We may assume that $X_1^\circ$ and $X_2^\circ$ are full by Lemma \ref{lem_nonfull}.
Then, $\lK(X_k^\circ)$ and $D$ are $KK$-equivalent,
and thus $\lK(X_k^\circ), J_\frakX$ and $\lK(\calF(\frakX)J_\frakX)$ are $K$-nuclear.
It follows from \cite[Proposition 3.8]{Skandalis2} that $B_k$ has the same property.
Since $\Pi$ induces a $KK$-equivalence, $\calT(\frakX)$ is $K$-nuclear,
and thus so is $\DT(A,E) \cong \calO(\frakX)$ by \cite[Proposition 3.8]{Skandalis2}.
Therefore, $\phi \otimes_{\DT(A,E)} \alpha$ is implemented by some nuclear Kasparov bimodule,
and hence the $K$-nuclearity of $A$ follows from Lemma \ref{lem_subequiv}.
\end{proof}

\end{document}